\documentclass[10pt]{article}
\pdfoutput=1
\RequirePackage[left=28mm,right=28mm,top=35mm,bottom=30mm]{geometry}
\usepackage[utf8]{inputenc}
\usepackage[english]{babel}
\usepackage{amsmath}
\usepackage{amssymb}
\usepackage{amsthm}
\usepackage[shortlabels]{enumitem} 
\usepackage[numbers]{natbib}
\usepackage{afterpage}
\usepackage{graphicx} 
\usepackage{tikz-cd} 
\usepackage{stmaryrd} 
\usepackage{url}
\usepackage{hyperref}

\newcommand{\End}{\operatorname{End}}
\newcommand{\Hom}{\operatorname{Hom}}
\newcommand{\add}{\!\operatorname{add}}
\newcommand{\domdim}{\operatorname{domdim}}
\newcommand{\m}{\!\operatorname{-mod}} 
\newcommand{\proj}{\!\operatorname{-proj}}

\newtheorem{numberingthm}{Theorem} 
\theoremstyle{definition}
\newtheorem{Def}[numberingthm]{Definition}
\newtheorem{example}[numberingthm]{Example}

\theoremstyle{plain}
\newtheorem{Prop}[numberingthm]{Proposition}
\newtheorem{Theorem}[numberingthm]{Theorem}
\newtheorem{Cor}[numberingthm]{Corollary}
\newtheorem{Lemma}[numberingthm]{Lemma}
\newtheorem{Remark}[numberingthm]{Remark}

\newenvironment{Example}
{\pushQED{\qed}\example}
{\popQED\endexample}

\providecommand{\keywords}[1]
{\scriptsize
	\textbf{\textit{Keywords:}} #1 \normalsize \hfill
}
\providecommand{\msc}[1]
{\scriptsize
	\textbf{\textit{2010 Mathematics Subject Classification:}} #1 \normalsize \hfill
}

\title{A characterisation of Morita algebras in terms of covers}
\author{Tiago Cruz} 
\date{}

\newcommand{\Address}{{
		\bigskip
		\footnotesize
		
		TIAGO CRUZ,\par \textsc{Institute of Algebra and Number Theory}\par \textsc{University of Stuttgart,}\par \textsc{Pfaffenwaldring 57, 70569 Stuttgart, Germany,}\par\nopagebreak
		\textit{E-mail address}, T.~Cruz: \texttt{tiago.cruz@mathematik.uni-stuttgart.de}
		}}
		
		\allowdisplaybreaks

\begin{document}

\maketitle

\begin{abstract}
A pair $(A, P)$ is called a cover of $\End_A(P)^{op}$ if the Schur functor $\Hom_A(P, -)$ is fully faithful on the full subcategory of projective $A$-modules, for a given projective $A$-module $P$. By definition, Morita algebras are the covers of self-injective algebras and then $P$ is a faithful projective-injective module. Conversely, we show that $A$ is a Morita algebra and $\End_A(P)^{op}$ is self-injective whenever $(A, P)$ is a cover of $\End_A(P)^{op}$ for a faithful projective-injective module $P$.
\end{abstract}

\subsection{Introduction}

Morita algebras were introduced in \citep{zbMATH06264368} to better understand and generalize self-injective algebras. The definition is based on a theorem by Morita (see \citep[section 16]{zbMATH03131967}, \citep[p. 185]{zbMATH06264368}) and it says that a Morita algebra is the endomorphism algebra of a generator over a self-injective algebra. Moreover, Morita showed that this generator can be chosen to be projective-injective of the form $Ae\simeq D(eA)$ when regarded as a left module over the Morita algebra $A$, for some idempotent $e$ of $A$. Modules containing the regular module as a direct summand are examples of generators.

Morita algebras occur in several contexts, including cover theory and the Morita-Tachikawa correspondence.

A cover, in Rouquier's sense \cite{zbMATH05503141}, of an algebra $B$ is a pair $(A, P)$ consisting of the endomorphism algebra $A$ of a generator over $B$ and a certain projective $A$-module $P$. Covers are useful to transfer properties from the cover to $B$ through a Schur functor $\Hom_A(P, -)$. This construction allows us to view the module category of $B$ as a kind of quotient of the module category of its cover $A$. It follows from their definition that Morita algebras are exactly the covers of self-injective algebras.

On the other hand, generators over self-injective algebras are also cogenerators. The endomorphism algebras of generators-cogenerators are described by the Morita-Tachikawa correspondence, which classifies the finite-dimensional algebras with dominant dimension at least two as the endomorphism algebras of a generator-cogenerator.
The famous Nakayama conjecture claims that finite-dimensional algebras with infinite dominant dimension are self-injective.

Many interesting covers arise as endomorphism algebras of generators-cogenerators. In this situation, the following questions arise. Given a faithful projective-injective $A$-module $P$:
\begin{itemize}
	\item When is $(A, P)$ a cover of $\End_A(P)^{op}$?
	\item When is $A$ a Morita algebra?
	\item When is $\End_A(P)^{op}$ a self-injective algebra?
\end{itemize}
Our main result provides answers to these questions and it provides several characterisations of Morita algebras with fewer assumptions than the theorem by Morita that motivated the definition of Morita algebras in \citep[pages 185-186]{zbMATH06264368}:
\begin{Theorem}\label{gendocharaccovers}
	Let $A$ be a finite-dimensional algebra. Assume that $P$ is a faithful projective-injective left $A$-module. Then the following assertions are equivalent:\begin{samepage}
		\begin{enumerate}[(i)]
		\item $(A, P)$ is a cover of $\End_A(P)^{op}$;
		\item $A$ is a Morita algebra;
		\item The endomorphism algebra $\End_A(P)^{op}$ is a self-injective algebra and $\domdim A\geq 2$;
		\item $\domdim A \geq 2$ and $\add_A DA\otimes_A P=\add_A P$.
		\item $\domdim A\geq 2$ and the Nakayama functor restricts to $DA\otimes_A -\colon \add_A P\rightarrow \add_A P$.
	\end{enumerate} \end{samepage}
\end{Theorem}
The implications $(ii)\Leftrightarrow (iii)\implies (i)$ are already known by  \citep[section 16]{zbMATH03131967} and Morita-Tachikawa correspondence.  The equivalence $(v)\Leftrightarrow(iv)\Leftrightarrow (ii)$ is related to the study of strongly projective modules in \cite{Onderivedequivalencesandhomologicaldimensions}. Here, we present a shorter proof.
The proof of Theorem \ref{gendocharaccovers} involves the study of double centralizer properties and a reformulation of the definition of Morita algebras using the Nakayama functor. Prominent examples of double centralizer properties are Soergel's double centralizer theorem \citep{zbMATH00005018}, classical Schur--Weyl duality \citep{green} and its many generalizations (see for example \citep{zbMATH07098063}).

As a byproduct of Theorem \ref{gendocharaccovers}, we clarify in Remark \ref{idempotentsdcp} some situations where a double centralizer property on a module $Ae$ is equivalent to a double centralizer property on $eA$, for some idempotent $e$ of a given finite-dimensional algebra $A$. Further, although it does not come as a surprise, we see in Example \ref{example15} that if $P$ is only projective the assertion $(i)$ together with $(A, \Hom_A(P, A))$ being a cover of $\End_A(P)^{op}$ is not sufficient for $A$ to be a Morita algebra.

As application of Theorem \ref{gendocharaccovers}, we give in Corollary \ref{qf1selfinjective} a criterion for a \emph{QF-1} algebra to be a self-injective algebra.

\subsection{Notation}

We will assume throughout this paper that $k$ is a field and $A$ and $B$ are finite-dimensional $k$-algebras.
By $A\m$ we mean the category of finitely generated  left $A$-modules and by $A\proj$ the full subcategory of $A\m$ whose modules are the finitely generated projective $A$-modules. We denote by $\add_A M$ (or just $\add M$ when $A$ is fixed) the full subcategory of $A\m$ whose modules are direct summands of finite direct sums of $M\in A\m$. We write $A\proj$ to denote $\add A$.
For any $M\in A\m$ and $f, g\in \End_A(M)$ the multiplication $fg$ is the composite $f\circ g$ of $g$ and $f$. The opposite algebra of $A$ will be denoted by $A^{op}$. 

Given a finitely generated $(A, B)$-bimodule $M$, there is a \emph{double centralizer property} on $M$ between $A$ and $B$ provided that the multiplication maps on $M$ induce isomorphisms $A\simeq \End_B(M)$ and $B\simeq \End_A(M)^{op}$. 
 By the standard duality $D$ we mean the functor $\Hom_k(-, k)\colon A\m\rightarrow A^{op}\m$.

An algebra $B$ is called \emph{self-injective} if there exists a $B$-isomorphism $DB\simeq B$. If there exists a $(B, B)$-bimodule isomorphism between $DB$ and $B$ then $B$ is called a \emph{symmetric algebra}.
\subsection{Dominant dimension}

Let 
\begin{align}
0\rightarrow {}_A A\rightarrow I_0\rightarrow I_1\rightarrow \cdots \rightarrow I_n \rightarrow \cdots
\end{align} be a minimal injective resolution of the regular module ${}_A A$. We say that the \emph{dominant dimension} of the algebra $A$, denoted by $\domdim A$, is $n\in \mathbb{N}\cup \{\infty \}$ if $I^t$ is projective for $t<n$ and $I_n$ is not. In particular, $\domdim A$ is infinite if all injective modules $I_t$ are projective. Analogously, we can define the dominant dimension using the right regular module $A_A$. This (right) dominant dimension is equal to $\domdim A$.
A detailed account on dominant dimension can be found in \citep{zbMATH03425791, zbMATH03248955}.
 $A$ is called \emph{QF-3 algebra } if $\domdim A\geq 1$. In such a case, $I_0$ is a faithful projective-injective module. Moreover, given another faithful projective-injective module $X\in A\m$, $\add X=\add I_0$ \citep[Lemma 2.3]{zbMATH01623143} and $\domdim A\geq n$ if there exists an exact sequence
\begin{align}
0\rightarrow A\rightarrow X_0\rightarrow X_1\rightarrow \cdots\rightarrow X_{n-1},
\end{align} where $X_i\in \add X$, $i=0, \ldots, n-1$. This last claim follows from \citep[7.7]{zbMATH03425791}. In particular, there exists an idempotent $e$ such that $Ae$ is a projective-injective faithful module which is a direct sum of pairwise non-isomorphic indecomposable modules. Under these conditions, $Ae$ is called a \emph{projective-injective minimal faithful} module.
Furthermore, a minimal faithful projective-injective module $Ae$ (if it exists) has a double centralizer property if and only if $\domdim A\geq 2$ (see for example \citep[Theorem 2]{zbMATH03248955}).
A module $M\in A\m$ is called a \emph{generator} if ${}_A A\in \add_A M$. Analogously, a module \mbox{$M\in A\m$} is called a \emph{cogenerator} if $DA\in \add_A M$. For self-injective algebras the notions of generator and cogenerator coincide.

	\begin{Theorem}[Morita-Tachikawa correspondence] \citep[Theorem 2]{zbMATH03248955}\label{MoritaTachikawa} There is a bijection:
	$$\left. \left\{
	\begin{array}{@{}c@{}}(B, M)\colon \begin{array}{@{}c@{}} B \text{ finite dimensional }\\
	\text{$k$-algebra}
	\\ \text{M a $B$-generator-cogenerator}
	\end{array} \end{array}\right\} \right/ \!\!\sim_1  \longleftrightarrow \left. \left\{ \begin{array}{@{}c@{}}
	A\colon
	\begin{array}{@{}c@{}}
	A  \text{ finite dimensional }\\ \text{$k$-algebra
} \\\domdim A\geq 2 \end{array} \end{array}
	\right\}    \right/  \!\!\sim_2
	$$
Here, $A\sim_2 A'$ if and only if $A$ and $A'$ are isomorphic, whereas, $(B, M)\sim_1 (B', M')$ if and only if there is an
	equivalence of categories $F\colon B\m\rightarrow B'\m$ such that $M'=FM$.
	\begin{flalign*}
\hspace{10em}	(B, M) \ \mapsto \ &  A=\End_B(M)^{op} \\
	(\End_A(N), N) \ \mapsfrom \ & A 
	\end{flalign*}where $N$ is a minimal projective-injective
	faithful right $A$-module.
\end{Theorem}

Usually, the Morita-Tachikawa correspondence is formulated for basic algebras. However, the above formulation is also equivalent due to a double centralizer property being a Morita invariant property.

\begin{Theorem}\citep[10.1]{zbMATH03425791}
	Let $A$ and $B$ be finite-dimensional $k$-algebras. Suppose that there is an equivalence $H\colon A\m\rightarrow B\m$. If there is a double centralizer property on $M\in A\m$ then there is a double centralizer property on $HM\in B\m$.
\end{Theorem}

 \subsection{Covers} The theory of covers was introduced by Rouquier \citep{zbMATH05503141}.
 
 \begin{Lemma}\citep[Proposition 4.33]{zbMATH05503141} Let $A$ and $B$ be finite-dimensional $k$-algebras such that $B=\End_A(P)^{op}$, for some $P\in A\proj$. Denote by $F$ the Schur functor $\Hom_A(P, -)\colon A\m\rightarrow B\m$ and denote by $G$ its right adjoint $\Hom_B(FA, -)$. The following assertions are equivalent. \label{dcpropertycover}
 	\begin{enumerate}[(i)]
 		\item The canonical map of algebras $A\rightarrow \End_B(FA)^{op}$, given by $a\mapsto (f\mapsto f(-)a)$, $a\in A, f\in FA$, is an isomorphism of $k$-algebras.
 		\item For all $M\in A\proj$, the unit $\eta_M\colon M\rightarrow GFM$ is an isomorphism of $A$-modules.
 		\item The restriction of $F$ to $A\proj$ is full and faithful.
 	\end{enumerate}
 \end{Lemma}

\begin{Def}\label{coverdef}
	We say that $(A, P)$ is a \textbf{cover} of $B$ if the restriction of \mbox{$F=\Hom_A(P, -)\colon A\m\rightarrow B\m$} to $A\proj$ is full and faithful. 
\end{Def}
 
 \begin{Remark}
 	In the situation of Definition \ref{coverdef}, a double centralizer property holds on $FA$, but not necessarily on $P$.
 \end{Remark}

Before we proceed with some basic results about covers, recall the following result.

\begin{Prop}\label{additiveclosureonendomorphismalg}
	Let $M, N\in A\m$ with $\add_A M=\add_A N$. Then the algebras $B:=\End_A(M)^{op}$ and $C:=\End_A(N)^{op}$ are Morita equivalent and the algebras $\End_B(M)$ and $\End_C(N)$ are isomorphic.
\end{Prop}
\begin{proof}
	See for example \citep[Proposition 1.3]{zbMATH06264368}.
\end{proof}

\begin{Prop}
	Let $A$ be a QF-3 algebra with a projective-injective faithful right $A$-module $V$. If $\domdim A\geq 2$ then $(A, \Hom_A(V, A))$ is a cover of $B:=\End_A(V)$.\label{dominantgeqtwodcp}
\end{Prop}
\begin{proof}
	Let $eA$ be the minimal right projective-injective faithful $A$-module. Since $\domdim A\geq 2$ there is a double centralizer property $\End_{eAe}(eA)^{op}\simeq A$. Because of $V$ being faithful projective-injective, $\add_A V=\add_A eA$. By Proposition \ref{additiveclosureonendomorphismalg}, $\End_{eAe}(eA)^{op}\simeq \End_B(V)^{op}$. Thus, 
	\begin{align}
	A\simeq \End_{eAe}(eA)^{op}\simeq \End_B(V)^{op}\simeq \End_B(\Hom_A(\Hom_A(V, A), A))^{op}.
	\end{align}The last isomorphism follows from $V$ being right $A$-projective and therefore $V$ being reflexive, that is, $V\simeq \Hom_A(\Hom_A(V, A), A)$. Further, this isomorphism is also an isomorphism of $B$-modules. So, the claim follows.
\end{proof}

The definition of cover can be formulated in general for finitely generated projective algebras over Noetherian rings. Unlike the general case, covers of finite-dimensional algebras can always be reduced to covers arising from idempotents.
 
 	\begin{Prop}\label{idempotentcovers}
 	If $(A, P)$ is a cover of $B$ then there exists an idempotent $e\in A$ such that $(A, Ae)$ is a cover of $eAe$.
 \end{Prop}
 \begin{proof}
 	We can decompose $P$ into a direct sum of projective indecomposables $P_1\oplus \cdots \oplus P_n$. By the  Krull-Remak-Schmidt Theorem, there is a subset $I$ of $\{1, \ldots, n\}$ so that $Q:=\bigoplus_{i\in I}P_i$ is an $A$-summand of $A$, where the modules $P_i$, $i\in I$, are pairwise non-isomorphic and $\add Q=\add P$.
 	 Moreover, there exists an idempotent $e\in A$ such that $Ae\simeq Q$. Hence, the algebras $B$ and $eAe$ are Morita equivalent. The functor $\Hom_B(\Hom_A(P, Ae), -)\colon B\m\rightarrow eAe\m$ is an equivalence of categories. On the other hand, the canonical map $\Hom_A(Ae, A)\rightarrow \Hom_B(F(Ae), FA)$ is bijective. Moreover, it is an $eAe$-isomorphism. Therefore, \begin{align}
 	A&\simeq \End_B(\Hom_A(P, A))^{op}\simeq \End_{eAe}(\Hom_B(\Hom_A(P, Ae), \Hom_A(P, A)) )^{op} \label{eqfc01} \\&= \End_{eAe}(\Hom_B(F(Ae), FA))^{op}\simeq \End_{eAe}(\Hom_A(Ae, A))^{op}. \tag*{\qedhere}
 	\end{align}
 \end{proof}

As mentioned, covers can be used to obtain properties of the module category of an algebra using one of its covers, for example, the number of blocks, or classification of simple modules, among many others. Although we do not pursue this direction here, cover theory really shines when the cover has finite global dimension and the algebra $B$ has not. For self-injective algebras $B$, covers of $B$ with finite global dimension are the non-commutative resolutions of \citep{zbMATH06439494}. As in their particular case, covers are non-commutative unless the cover of $B$ is isomorphic to $B$ itself.

\begin{Prop}
	Suppose that $A$ is a finite-dimensional commutative $k$-algebra. If $(A, Ae)$ is a cover of $eAe$, for some idempotent $e$ in $A$, then $A$ is isomorphic to $eAe$.
\end{Prop}
\begin{proof}
	The commutativity of $A$ implies that $e$ is a central idempotent and $eAe$ is commutative. If $(A, Ae)$ is a cover of $eAe$ then 
	\begin{align}
	A\simeq \End_{eAe}(eA)=\End_{eAe}(e^2A)=\End_{eAe}(eAe)\simeq eAe. \tag*{\qedhere}
	\end{align}
\end{proof}

\subsection{Morita algebras and Nakayama functor}

Morita algebras were introduced by Kerner and Yamagata in \citep{zbMATH06264368}. A finite-dimensional $k$-algebra $A$ is called a \emph{Morita algebra} if it can be written as the endomorphism ring of a generator-cogenerator over some self-injective algebra. A detailed account on Morita algebras and double centralizer properties can also be found in \citep{zbMATH06389727}.
A characterization of dominant dimension over Morita algebras in terms of cohomology over self-injective algebras was given in \citep{zbMATH06814513}.

For the proof of the main result, we require the following characterisation of Morita algebras. Theorem \ref{relativeMorita} is an extension of Proposition 2.9 of \citep{Onderivedequivalencesandhomologicaldimensions} (formulated in a different terminology).

\begin{Theorem}\label{relativeMorita}
	Let $A$ be a QF-3 $k$-algebra. Let $P$ be a faithful projective-injective left $A$-module.
	The following assertions are equivalent. \begin{enumerate}[(a)]
		\item $\domdim A\geq 2$ and the Nakayama functor restricts to $DA\otimes_A -\colon \add P\rightarrow \add P$.
		\item   $\domdim A \geq 2$ and $\add_A DA\otimes_A P=\add_A P$.
		\item The endomorphism algebra $B=\End_A(P)^{op}$ is self-injective with generator $P\in \operatorname{mod}(B)$ and \mbox{$A\simeq \End_B(P)$}, that is, $A$ is a Morita algebra.
	\end{enumerate}
	\begin{enumerate}[(a')]
		\item $\domdim A\geq 2$ and the Nakayama functor restricts to $-\otimes_A DA\colon \add DP\rightarrow \add DP$.
		\item $\domdim A\geq 2$ and $\add_A DP\otimes_A DA=\add_A DP$.
	\end{enumerate}
\end{Theorem}
\begin{proof}
	We will show $(b)\implies (a)\implies (c)\implies (b)$. The implications $(b')\implies (a')\implies (c)\implies (b')$ are analogous.
	
	The implication $(b)\implies (a)$ is clear since $DA\otimes_A X\in \add DA\otimes_A P=\add P$, for all $X\in \add_A P$.
	
	Assume that $(a)$ holds. Write $B=\End_A(P)^{op}$.
	Let $Ae$ be a minimal faithful projective-injective module. Then $\add Ae=\add P$. By Proposition \ref{additiveclosureonendomorphismalg}, $\End_B(P)\simeq \End_{eAe}(Ae)$.
	By Morita-Tachikawa correspondence,
	\begin{align}
	\End_B(DP)^{op}\simeq \End_B(P)\simeq \End_{eAe}(Ae)\simeq A,
	\end{align} and $Ae$ is a generator of $eAe$. Since equivalence of categories preserves generators, $P$ is a generator of $B$. 
	 It remains to show that $B$ is  self-injective. This follows by observing that
	\begin{align}
	B=\Hom_A(P, P)\simeq \Hom_A(P, A)\otimes_A P\simeq D(DA\otimes_A P)\otimes_A P\in \add DP\otimes_A P=\add DB.
	\end{align}
	Hence $B$ is $B$-injective. 
	
	Finally, assume that $(c)$ holds.
	Let $Ae$ be a minimal faithful projective-injective module. Again, since $\add_A Ae=\add_A P$, $eAe$ is Morita equivalent to $B$. So $Ae$ is a generator of $eAe$ and $A\simeq \End_B(P)\simeq \End_{eAe}(Ae)$. By Morita-Tachikawa correspondence,
	$\domdim A\geq 2$. Since $A\simeq \End_B(P)$ there exists an $(A, A)$-bimodule isomorphism $DA\simeq P\otimes_B DP$. Moreover, as left $A$-modules,
	\begin{align}
	DA\otimes_A P\simeq P\otimes_B DP\otimes_A P\simeq P\otimes_B DB.
	\end{align}
	Since $DB$ is $B$-projective and $B\in \add DB$, $DB$ is a $B$-progenerator. Hence, $\add_A DA\otimes_A P=\add_A P$. This completes the proof.
\end{proof}
 
 Using the terminology of \citep{Onderivedequivalencesandhomologicaldimensions}, Theorem \ref{relativeMorita} says that all faithful projective-injective modules over a Morita algebra are strongly projective-injective. In particular, this provides a new and shorter proof for Proposition 2.9 of \citep{Onderivedequivalencesandhomologicaldimensions}.
 
\subsection{Proof of the Main Theorem}

\begin{proof}[Proof of Theorem \ref{gendocharaccovers}]
The equivalence $(ii)\Leftrightarrow (iii)$ follows from the definition of Morita algebras and the Morita-Tachikawa correspondence. The equivalence $(ii)\Leftrightarrow (iv)\Leftrightarrow(v)$ is the content of Theorem \ref{relativeMorita}.
	 
	Assume that $A$ is a Morita algebra. 
	By Theorem \ref{relativeMorita}, $\add DA\otimes_A P=\add P$. Let $Ae$ be a minimal projective-injective faithful module.
	Then $\add_A \Hom_A(P, A)=\add_A DP=\add D(Ae)$. Since $\domdim A\geq 2$, we can write
	\begin{align}
	A\simeq \End_{eAe}(Ae)\simeq \End_{eAe}(D(Ae))^{op}\simeq \End_B(\Hom_A(P, A))^{op}.
	\end{align}
	 This shows that $(A, P)$ is a cover of $B$.

	Conversely, suppose that $(A, P)$ is a cover of $B:=\End_A(P)^{op}$.
	By Lemma \ref{dcpropertycover}, there is a double centralizer property on $\Hom_A(P, A)$. More precisely,
	\begin{align}
	\End_A(\Hom_A(P, A))\simeq B \quad \End_B(\Hom_A(P, A))^{op}\simeq A. \label{eqfc100}
	\end{align} In particular, $\Hom_A(P, A)$ is faithful-projective as right $A$-module. Hence, there exists an injective $A$-homomorphism
	$A\rightarrow \Hom_A(P, A)^s$, for some $s>0$.  Since $DP$ is projective as right $A$-module, there is a monomorphism $DP\rightarrow A^t\rightarrow \Hom_A(P, A)^{st}$. $DP$ is injective as right $A$-module. Hence, $DP\in \add_A \Hom_A(P, A)$.
	
	We claim now that $DA\otimes_A P$ is a left $A$-projective module. To see this, define $P'$ to be the direct sum of all non-isomorphic indecomposable $A$-modules that belong to the additive closure of $P$. So, $\add P=\add P'$ and $P'\in \add_A DA\otimes_A P=\add_A DA\otimes_A P'$. By Krull-Remak-Schmidt theorem, we can write $DA\otimes_A P'\simeq P'\oplus X$, for some $A$-module $X$. 
	On the other hand,
	\begin{align}
	\End_A(P'\oplus X)^{op}\simeq \End_A(DA\otimes_A P')^{op}\simeq \End_A(\Hom_A(P', A))\simeq \End_A(P')^{op}.
	\end{align}
	So, by comparing $k$-dimensions, $X$ must be the zero module. Hence, $DA\otimes_A P'$ is a faithful projective-injective  module. Consequently, $DA\otimes_A P$ is also a faithful projective-injective  module. 
	Now, the double centralizer property (\ref{eqfc100}) implies that $\domdim A\geq 2$. Since both $P$ and $DA\otimes_A P$ are  faithful projective-injective  modules, $\add_AP=\add_A DA\otimes_A P$. So, $A$ is a Morita algebra by Theorem \ref{relativeMorita}.
\end{proof}	

\begin{Remark}\label{idempotentsdcp}
	For a idempotent $e$ of $A$, $\Hom_A(Ae, A)\simeq eA$ as $(eAe, A)$-bimodules. In addition, assume that $e$ is an idempotent so that $Ae$ or $eA$ is projective-injective. By Theorem \ref{gendocharaccovers}, 
	a double centralizer property on $Ae$ is not equivalent to a double centralizer property on $eA$ unless $A$ is a Morita algebra.
	 In particular, if $A$ is a Morita algebra then $A=\End_{eAe}(eA)^{op}=\End_{eAe}(Ae)$.
\end{Remark}

\subsection{An application and two examples}

A finite-dimensional $k$-algebra is called \emph{QF-1 algebra} if all faithful $A$-modules have the double centralizer property (see \citep{zbMATH03061631}).

\begin{Cor}\label{qf1selfinjective}
	Let $A$ be a QF-1 $k$-algebra. Assume that $P$ is a faithful projective-injective left $A$-module.
	Then, $A$ is self-injective if and only if $\Hom_A(P, A)$ is a faithful right $A$-module.
\end{Cor}
\begin{proof}
	One direction is clear. Assume that $\Hom_A(P, A)$ is faithful. Since $A$ is a QF-1 algebra, $(A, P)$ is a cover of $\End_A(\Hom_A(P, A))\simeq \End_A(P)^{op}$. By Theorem \ref{gendocharaccovers}, $A$ is a Morita algebra. By \citep[Proposition 2.2]{zbMATH07042781}, a Morita algebra is QF-1 if and only if it is self-injective. Therefore, $A$ is self-injective.
\end{proof}

\begin{Example}
	\emph{For a QF-3 algebra $A$ with dominant dimension two and with a projective-injective faithful module $P$ the pair $(A, P)$ is not, in general, a cover of $\End_A(P)^{op}$.}
	
	Let $k$ be an algebraically closed field. \label{faithfulprojinjnotgivesacover}
	Let $A$ be the following bound quiver $k$-algebra 
	
	\begin{center}
		\begin{tikzcd}
			1 \arrow[r, "\alpha_1"] & 2 \arrow[r, "\alpha_2"]& 3,
		\end{tikzcd} $\alpha_2\alpha_1=0 $.
	\end{center} 
	Note that we read the arrows in a path like morphisms, that is, from right to left.
	
	Denote by $P(i)$ the projective indecomposable module associated with the vertex $i$ and denote by $I(i)$ the indecomposable injective module associated with the vertex $i$. 
	
	The indecomposable projective (left) modules are given by
	\begin{equation}
		P(1) = I(2) = \begin{array}{c}
			1 \\
			2
		\end{array}, \quad P(2)=I(3)=\begin{array}{c}
			2 \\
			3
		\end{array}, \quad P(3)=\begin{array}{c}
			3
		\end{array}.
	\end{equation} 
	\begin{align}
		0\rightarrow A\rightarrow P(1)\oplus P(2)\oplus P(2)\rightarrow P(1)\rightarrow I(1)\rightarrow 0
	\end{align}is a minimal injective resolution of $A$. Denote by $P$ the projective module $P(1)\oplus P(2)$. Hence, $A$ is a QF-3 algebra with minimal faithful projective-injective left $A$-module $P$ and with $\domdim A\geq 2$.  Here $B=\End_A(P)^{op}$ is the path algebra with quiver
	\begin{center}
		\begin{tikzcd}
			1 \arrow[r, "\alpha_1"] & 2.
		\end{tikzcd} 
	\end{center} But $B$ is not self-injective. By Theorem \ref{gendocharaccovers}, $(A, P)$ is not a cover of $B$. But, $(A, P(2)\oplus P(3))$ is a cover of $\End_A(P)^{op}$ by Proposition \ref{dominantgeqtwodcp}. In fact, $P(2)\oplus P(3)\simeq \Hom_A(DA, P)=\Hom_A(DP, A)$ as left $A$-modules.
\end{Example}

This example also shows that $\End_B(\Hom_A(P, A))^{op}$ is not isomorphic to $\End_B(P)$, in general.

\begin{Remark}
	If we drop the injectivity of $Ae$ and of $eA$ in Remark \ref{idempotentsdcp}, the statement is false. This can be seen in the next example.
\end{Remark}

\begin{Example}\label{example15}
\emph{There are idempotents $e$ and non-Morita algebras $A$ so that there are  double centralizer properties on $Ae$ and on $eA$.}

Let $k$ be an algebraically closed field. Let $A$ be the following bound quiver $k$-algebra
\begin{equation}
	\begin{tikzcd}[every arrow/.append style={shift left=0.75ex}]
		1 \arrow[r, "\alpha"] & 2 \arrow[l, "\beta"] \arrow[r, "\gamma"] & 3\arrow[l, "\theta"]
	\end{tikzcd}, \quad \gamma\alpha=\beta\theta=\alpha\beta=\gamma\theta=0.
\end{equation}We are using the same notation as in the previous example. So, the indecomposable projective (left) modules are given by 
\begin{equation}
	P(1)  = Ae_1= \begin{array}{c}
		1 \\2\\1
	\end{array}, \quad P(2)=Ae_2=\begin{array}{ccc}
		& 2 &\\ 1& & 3 \\ & & 2
	\end{array}, \quad P(3)=Ae_3=\begin{array}{c}
		3\\2
	\end{array}.
\end{equation} The projective $P(3)$ has dominant dimension zero so $A$ cannot be a Morita algebra. We can see that $A$ has an involution fixing the primitive idempotents and interchanging $\alpha$ with $\beta$ and $\gamma$ with $\theta$. Fix $e=e_1+e_2$. By a direct computation, we can see that $(A, P(1)\oplus P(2))$ is a cover of $B=eAe$. Here, $B$ is the bound quiver $k$-algebra 
\begin{equation}
	\begin{tikzcd}[every arrow/.append style={shift left=0.75ex}]
		1 \arrow[r, "\alpha"] & 2 \arrow[l, "\beta"] \arrow[loop right, "t", looseness=15]
	\end{tikzcd}, \quad \alpha\beta=\beta t=t\alpha = 0.
\end{equation}Again by a direct computation or by observing that the duality of $A$ restricts to one of $B$ fixing $e$ it follows that $\End_{eAe}(Ae)\simeq \End_{eAe}(eA)^{op}\simeq A$.
\end{Example}

\section*{Acknowledgments}
The results of this paper are contained in the author's forthcoming PhD thesis, financially supported by \textit{Studienstiftung des Deutschen Volkes}. The author would like to thank Steffen Koenig for all his comments and suggestions towards improving this manuscript.

\bibliographystyle{plainnat}

\Address
\end{document}